\documentclass[a4paper]{article}

\usepackage{mathrsfs}
\usepackage{amsfonts}
\usepackage{amsmath}
\usepackage{amssymb}
\usepackage{amstext}
\usepackage{amsthm}
\usepackage{graphicx}

\newcommand{\R}{\mathbb R}
\newcommand{\N}{\mathbb N}

 \newtheorem{teo}{Theorem}[section]
 
 \newtheorem{lemma}[teo]{Lemma}
 \newtheorem{prop}[teo]{Proposition}
 \theoremstyle{definition}
 \newtheorem{defin}[teo]{Definition}
 \theoremstyle{remark}
 \newtheorem{rem}[teo]{Remark}
 \newtheorem*{ex}{Example}

\DeclareMathOperator*{\liminj}{lim_{\rightarrow}}
\DeclareMathOperator{\mult}{mult} \DeclareMathOperator{\Hom}{Hom}

\begin{document}

\title{Morse Theory for Geodesics in Conical Manifolds}
\author{Marco G. Ghimenti\thanks{
Dipartimento di Matematica, Universit\`a di Pisa, via Buonarroti 1c, 56127, Pisa, Italy}
}
\date{}
\maketitle

\begin{abstract}
The aim of this paper is to extend the Morse theory for geodesics
to the conical manifolds. In a previous paper \cite{Ghi05} we
defined these manifolds as
submanifolds of $\R^n$ with a finite number of conical
singularities. To formulate a good Morse theory we use an
appropriate definition of geodesic, introduced in the cited work. The
main theorem of this paper (see theorem \ref{main},
section \ref{morsesec}) proofs
that, although the energy is nonsmooth, we can find
a continuous retraction of its sublevels in absence of
critical points. So, we
can give a good definition of index for isolated critical values
and for isolated critical points. We
prove that Morse relations hold and, at last, we give a definition
of multiplicity of geodesics which is geometrical
meaningful. In section \ref{sezconfronto} we compare our theory with
the weak slope approach existing in literature. Some examples
are also provided.

\end{abstract}
\section{Introduction}

In a previous work \cite{Ghi05} we introduced conical manifolds as
submanifolds of $\R^n$, giving the following definition:
\begin{defin}\label{defvarieta}
A conical manifold $M$ is a complete $n$-dimensional $C^0$ sub
manifold of $\R^m$ which is everywhere smooth, except for a finite
set of points $V$. A point in $V$ is called vertex.
\end{defin}

These kind of manifold have many common features with other
singular manifolds present in literature, as the {\em
Cone-manifolds\/} studied by Hodgson and Tysk \cite{HT93}, the
{\em piecewise linear manifolds\/} (see, e.g
\cite{Ban67,Sto73,Sto76}), or the {\em Orbifolds\/}
(\cite{BL96,Bor92,Thu78}); in \cite{Ghi04} these links are briefly
examined.

Furthermore, the geodesic problem on a singular manifold is
studied by Degiovanni and Morbini, in \cite{DM99}. In section
\ref{sezconfronto} we will point out the common feature and the
main differences between our approach and the Degiovanni and
Morbini one.

In our previous work we gave the following definition of geodesic,
which seems to be appropriate for these kind of problems
(\cite[Definition 2]{Ghi05}).
\begin{defin}
\label{defgeod} A path $\gamma\in H^{1,2}([0,1],M)$ is a geodesic
iff
\begin{itemize}
\item{ the set $T=T_\gamma:=\{s\in(0,1):\gamma(s)\in V\}$ is a
closed set without internal part;} \item $D_s\gamma'=0\,\,\forall
s\in[0,1]\smallsetminus T$; \item $|\gamma'|^2$ is constant as a
function in $L^1$.
\end{itemize}
\end{defin}
\noindent By this definition, we were able to prove a deformation
lemma that allows us to apply the main theorems of the calculus of
variation in the large. In particular we can estimate the number
of geodesics using the Ljusternik-Schnirelmann theory
(\cite[Theorem 1]{Ghi05}).

In this work we will associate to each geodesic $\gamma$ an index
$i(\gamma)$ similar to the Morse index, that allows us to obtain
some information on the qualitative properties of the geodesics.
By this index we state also an analogous of the Morse relations
that hold in the smooth case.

As usual, given $p,q \in M$, we set
\begin{displaymath}
{\Omega}={\Omega}_{p,q}:= \left\{\gamma \in
W^{1,2}([0,1],\R^n)|\gamma(0)=p,\gamma(1)=q, \gamma([0,1])\subset
M \right\},
\end{displaymath}
the suitable path space, and a functional
\begin{equation}
E:{\Omega}\rightarrow\R\label{defen1}
\end{equation}
\begin{equation}
E(\gamma)=\int\limits_0^1\|\gamma'\|^2dt.\label{defen2}
\end{equation}
Moreover we set
\begin{displaymath}
{\Omega}^c={\Omega}^c_{p,q}:=\left\{ \gamma\in {\Omega},
E(\gamma)\leq c \right\}.
\end{displaymath}
\begin{displaymath}
{\Omega}^b_a=\{{\Omega}_{p,q}\}^b_a:=\left\{ \gamma\in {\Omega},
a\leq E(\gamma)\leq b \right\}.
\end{displaymath}
Under suitable assumption, if $M$ is a generic smooth manifold, we
know that the following Morse relations hold.
\begin{equation}
\label{morseclass} \sum_{K\cap{\Omega}_a^b}\lambda^{m(\gamma)}= {
P}_\lambda({\Omega}^a,{\Omega}^b)+(1+\lambda) {  Q}_\lambda.
\end{equation}
\begin{equation}
\label{morseclass2} \sum_{K}\lambda^{m(\gamma)}=
{P}_\lambda({\Omega})+(1+\lambda) {  Q}_\lambda.
\end{equation}
where $K$ is the set of geodesics, ${  P}_\lambda(X,Y)$ is the
Poincar\'e polynomial (in the variable $\lambda$) of the couple
$(X,Y)$, ${  Q}_\lambda$ is a formal series with coefficient in
$\N \cup \{+\infty\}$, and $m(\gamma)$ is the Morse index of
the geodesic $\gamma$, i.e. the signature of the second variation
of the energy. Furthermore, if we take a geodesic $\gamma$ s.t.
$E(\gamma)=c$ and
${\Omega}^{c+{\varepsilon}}_{c-{\varepsilon}}\cap K=\gamma$ for
some ${\varepsilon}>0$, then we have that
\begin{equation}\label{morseclass3}
\lambda^{m(\gamma)}={
P}_\lambda({\Omega}^{c+{\varepsilon}},{\Omega}^{c-{\varepsilon}})
\end{equation}
(for some references on Morse theory the reader can check, for
example \cite{Pal63}, \cite{Ben91,Ben95,BG94},\cite{Cha85}).

Coming back to conical manifold, the usual definition of
$m(\gamma)$ makes no sense since the energy is not differentiable.
However, under suitable assumptions, we can define the index of a
geodesic as the formal polynomial
\begin{equation}
i(\gamma)={
P}_\lambda({\Omega}^{c+{\varepsilon}},{\Omega}^{c-{\varepsilon}}),
\end{equation}
according to (\ref{morseclass3}). In this case the Morse relations
(\ref{morseclass}) and (\ref{morseclass2}) become
\begin{equation}
\label{morseconic} \sum_{K\cap{\Omega}_a^b}i(\gamma)= {
P}_\lambda({\Omega}^a,{\Omega}^b)+(1+\lambda) {  Q}_\lambda.
\end{equation}
\begin{equation}
\label{morseconic2} \sum_{K}i(\gamma)= {
P}_\lambda({\Omega})+(1+\lambda) {  Q}_\lambda.
\end{equation}
In this work we prove that $i(\gamma)$ is a good definition, i.e.
does not depends on ${\varepsilon}$ and that (\ref{morseconic})
and (\ref{morseconic2}) hold.

We have also that, while in a smooth manifold, for a generic
geodesic, $i(\gamma)$ is a monome in $\lambda$, for a geodesic in
a conical manifold this is not true. Then we have introduce the
multiplicity of a geodesic (Definition \ref{mult}) as
\begin{equation}
\mult({\gamma})=i(\gamma)|_{\lambda=1}.
\end{equation}
we show some examples in which  an high multiplicity, or a 0
multiplicity of a geodesic occurs.

Also, the definition of index that we give, allows us to make a
comparison between our approach to conical manifold and the
approach of Degiovanni, Marzocchi and Morbini (Section
\ref{sezconfronto}), which is based on the concept of the {\em
weak slope\/} (see \cite{CD94,CDM93,De96,DM94}), and was presented
in \cite{DM99,MM02}

\section{Preliminary Results}
In this section we present some peculiarities of the study of
geodesics in conical manifold, which motivate Definition
\ref{defgeod}. Also, we resume briefly some result contained in
\cite{Ghi05} which will be useful for this paper.

Usually there are two ways to introduce geodesics in a smooth
manifolds: at first we can formulate a Cauchy problem, i.e., given
$p\in M$, $v\in T_pM$, we look for a continuous curve
$\gamma:I\rightarrow M$ s.t.
\begin{equation}\label{PbCau}
\left\{
\begin{array}{l}
D_s\gamma'=0;\\
\gamma(0)=p;\\
\gamma'(0)=v.
\end{array}
\right.
\end{equation}
Otherwise we can choose a suitable path space and an energy
functional, for example, the space ${\Omega}$ and the functional
$E$ previously defined, and we look for critical points of the
energy.

Contrarily to the smooth case, these two methods are not
equivalent for manifolds which have conical singularities, and
each one of them have some peculiarity. If we try to solve the
Cauchy problem we have neither uniqueness nor continuous
dependence from starting condition: David Stone (see \cite{Sto73},
\cite{Sto76}) showed that there are piecewise linear manifold with
a vertex $p$ (indeed a special case of conical manifolds) in
which, given a point $q$, there is a family of minimal geodesics
solving (\ref{PbCau}) that are the same straight line between $q$
and $p$, then start again from $p$ with different angles. So
(\ref{PbCau}) has many solutions, and there is no reasonable
criterion to choose one of these geodesics instead of another.

\begin{figure}[h]
\begin{center}
\includegraphics[scale=.7,angle=270]{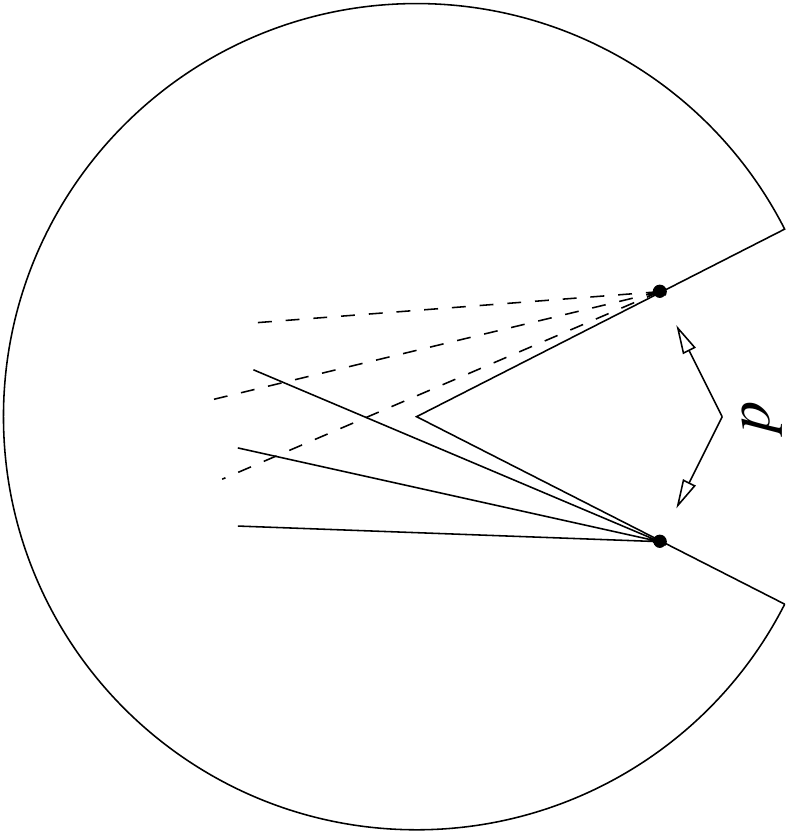}
\caption{Geodesics starting from a point.} \label{Figure1}
\end{center}
\end{figure}

Now, taken an Euclidean half cone, we can represent it as a
circular sector in a plane, in which the straight edges are
identified. Taking $p$ on these edges, we represent the geodesics
as straight lines starting from $p$. In Fig. \ref{Figure1} we show
two sequence of geodesics, respectively the continuous and the
dotted lines, approaching the edges of the circular sector. It is
easy to see that a dotted and a continuous line can not be too
close, even if they approach the geodesic which start from $p$ and
goes to the vertex. Even in a familiar case we are not able to say
how can a geodesic pass through the vertex, in order to have
continuous dependence from initial data.

Also, if we take two points, say $p$ and $q$, symmetrical respect
to the vertex, we can say that there are always two minimal
geodesics joining them, even if $p$ and $q$ are as close as we
want; so it is not possible to define a normal neighborhood of the
vertex. Thus we can not hope to define a global exponential map.

The functional approach, on the contrary, gives us an easy result
on minimal geodesics. To finding minima is useful to extend $E$ in
the whole space

\begin{equation}
W^{1,2}_{p,q}: =\left\{\gamma \in W^{1,2}([0,1],\R^n)|\
\gamma(0)=p,\gamma(1)=q\right\}
\end{equation}
and to look at minima of $E$ with the constraint
$\text{Im}(\gamma)\subset M$; in this way we know that $E$ is
lower semi continuous for the $W^{1,2}$-weak topology and the set
$\big\{\gamma|\gamma\in {\Omega}\cap W^{1,2}_{p,q},\,
\|\gamma\|_{W^{1,2}}<c\big\}$ is weakly compact in
$W^{1,2}_{p,q}$, because $M$ is complete: then the problem
\begin{equation}
\min_{\substack{\gamma\in W^{1,2}_{p,q}\\\text{Im}\gamma\subset
M}} E(\gamma)
\end{equation}
has a solution, and  we can formulate the following theorem.

\begin{teo}
\label{min} Let $M$ be a conical manifold; let $p,q\in M$.

Then, the energy functional
\begin{equation}
\begin{array}c
E:{\Omega}_{p,q}\rightarrow\R\\
E(\gamma)=\int\limits_0^1\|\gamma'\|^2dt.
\end{array}
\end{equation}
has a minimum.

\end{teo}
\begin{proof}
Just proved.
\end{proof}
However, this approach is not completely useful, because the
energy is not smooth, so there is no easy way to define a critical
point of energy different from minimum. In the cited work we gave
a {\em generalized\/} definition of geodesic (Definition 2) which
seems to be the right one for these kind of problems.

In fact, although this is a local definition, it let us show two
deformation lemmas typical of the calculus of variation in the
large (\cite[Theorem 2 and Theorem 7]{Ghi05}). This lemmas are
crucial for this work and are briefly recalled here. In the next
section we will provide also some sketch of the proof of these
results.

\begin{lemma}[First deformation lemma]\label{deflem1}
Let $M$ be a conical manifold, and let $p,q\in M$. Suppose that
there exists $c>0$ s.t. there are a finite number of geodesics
with energy lesser than $c$. Suppose also that there are no
geodesics which energy is in the strip $[a,b]$ for some
$a,b\in\R$, $a<b<c$. Then ${\Omega}^ a$ is a deformation
retract of ${\Omega}^b$.
\end{lemma}
\begin{lemma}[Second deformation lemma]\label{deflem2}
Let $M$ be a conical manifold, $p,q\in M$. Suppose that there
exists $c>0$ s.t. there are a finite number of geodesics with
energy lesser than $c$. Suppose also that there are no geodesics
which energy is in the strip $[a,b)$ for some $a,b\in \R$,
$a<b<c$. Set $k_b$ the set of $b$-energy geodesics, then there
exists a neighborhood $U$ of $k_b$ s.t.
\begin{displaymath}
{\Omega}^b\smallsetminus U\simeq{\Omega}^a.
\end{displaymath}
\end{lemma}
These results link the topological structure of sublevels of $E$
with our definition of geodesics and allow us to consider
geodesics as critical points of energy. By these lemmas, in the
cited work we proved an estimate on the number of geodesic in a
conical manifold, via the Ljusternik-Schnirelmann theory. In this
paper these lemmas allow us to prove that the definition of the
index $i(\gamma)$ is a good definition.

\section{Morse Theory For Critical Levels}
\label{morsesec} In order to define the index of a critical point
and of a critical level, we recall the definition of Poincar\'e
polynomial
\begin{defin}\label{defpol}
Let $H^*$ the Alexander-Spanier cohomology with coefficient in a
field ${\mathbb F}$. For every pair $(X,A)$ of closed spaces, the
Poincar\'e polynomial is the formal series in the variable
$\lambda$ (with the convention that $\lambda^\infty=0$)
\begin{equation}
{  P}_\lambda(X,A,{\mathbb F})= {
P}_\lambda(X,A)=\sum_{q=0}^\infty\dim H^q(X,A,{\mathbb
F})\lambda^q.
\end{equation}
Moreover, we set
\begin{equation}
{  P}_\lambda(X,{\mathbb F}):={  P}_\lambda(X,\emptyset,{\mathbb
F}).
\end{equation}
\end{defin}
For the main properties of the formal series we refer to
\cite{BG94}. We use the Alexander-Spanier cohomology (see
\cite{Spa66} for an exhaustive treatment) to define the index of a
geodesic because we need its continuity property (recalled in
Theorem \ref{alex}) to prove Theorem \ref{mainpunt}.
\begin{defin}\label{ind}
Given $p,q \in M$, suppose that there exists a $\bar c $ s.t.
${\Omega}^{\bar c}$ contains a finite number of geodesics. Let
$c<\bar c$. Set, as usual $k_c$ the set of $c$-energy geodesic, we
define the index of $k_c$ as
\begin{equation}
i(k_c)={
P}_\lambda({\Omega}^{c+{\varepsilon}},{\Omega}^{c-{\varepsilon}}),
\end{equation}
for ${\varepsilon}$ sufficiently small.
\end{defin}
In order to prove that \ref{ind} is a good definition, we need to
show now that there exists $\bar{\varepsilon}$ s.t. $\forall
{\varepsilon}_1<{\varepsilon}_2\leq \bar{\varepsilon}$
\begin{displaymath}
{
P}_\lambda({\Omega}^{c+{\varepsilon}_1},{\Omega}^{c-{\varepsilon}_1})=
{
P}_\lambda({\Omega}^{c+{\varepsilon}_2},{\Omega}^{c-{\varepsilon}_2}).
\end{displaymath}

In order to prove it, we show that ${\Omega}^{c+{\varepsilon}_2}$
retracts on ${\Omega}^{c+{\varepsilon}_1}$ and that
${\Omega}^{c-{\varepsilon}_1}$ retracts on
${\Omega}^{c-{\varepsilon}_2}$. This is possible by the
deformation lemmas (\ref{deflem1} and \ref{deflem2}) cited
previously. Here we recall shortly the main tools of the proof.

We choose ${\varepsilon}$ s.t. $[c-{\varepsilon},c+{\varepsilon}]$
contains only $c$ as critical value. We start defining some
special subset of
${\Omega}^{c+{\varepsilon}_1}_{c+{\varepsilon}_2}$ (for
${\Omega}^{c-{\varepsilon}_2}_{c-{\varepsilon}_1}$ we will act in
the same way). Let
\begin{equation}
\Sigma=\{\gamma\in
{\Omega}^{c+{\varepsilon}_1}_{c+{\varepsilon}_2}, \text{ s.t.
Im}\gamma\cap V\neq\emptyset\}
\end{equation}
(we recall that $V$ is the set of vertexes). If $\gamma_i$ is a
geodesic in ${\Omega}^{c+{\varepsilon}}$, let
\begin{equation}
\Sigma_i=\{\gamma\in\Sigma\text{ s.t. }\gamma=\gamma_i \text{ up
to affine reparametrization }\}.
\end{equation}

\begin{lemma}\label{comp}
$\Sigma_i$ is compact for all $i$
\end{lemma}
\begin{proof}
The proof is not difficult and can be found in (\cite[Lemma
3]{Ghi05})
\end{proof}

By the compactness we can prove the following lemma.

\begin{lemma}[existence of retraction in $\bigcup_i\Sigma_i$]\label{retrsigma}
There exist $R\supset \bigcup_i\Sigma_i$, $\nu,\bar t\in\R^+$
and
\begin{displaymath}
\eta_R:R\times[0,\bar t\,]\rightarrow H^1(I,M)
\end{displaymath}
a continuous function s.t.

\begin{itemize}
\item $\eta_R(\beta,0)=\beta$ \item
$E(\eta_R(\beta,t))-E(\beta)<-\nu t$
\end{itemize}
for all $t\in [0,\bar t\,]$, $\beta\in R$.
\end{lemma}

\begin{proof}
The proof is quite technical, and is based on a modification of
Degiovanni and Marzocchi techniques. For the details we refer to
(\cite[Lemma 4]{Ghi05})
\end{proof}

We can prove also a deformation result that is in some sense the
complementary of the previous lemma.

\begin{lemma}\label{retru}
For any $U\supset\bigcup_i\Sigma_i$ there exist $\bar
t,\nu\in\R^+$ and a continuous functional
\begin{displaymath}
\eta_U:
{\Omega}^{c+{\varepsilon}_1}_{c+{\varepsilon}_2}\smallsetminus
U\times [0,\bar t\,] \rightarrow
{\Omega}^{c+{\varepsilon}_1}_{c+{\varepsilon}_2}
\end{displaymath}
such that
\begin{itemize}
\item $\eta_U(\cdot,0)=\text{Id}$ \item
$E(\eta_U(\beta,t))-E(\beta)\leq-\nu t$
\end{itemize}
for all $t\in[0,\bar t\,]$, for all $\beta \in
{\Omega}^{c+{\varepsilon}_1}_{c+{\varepsilon}_2}\smallsetminus U$
\end{lemma}

\begin{proof}
We use that for every $\gamma \in \Sigma \smallsetminus U$ we can
find a vector field $w_\gamma$ s.t. $dE(\gamma)[w_\gamma]$ exists,
although the energy is not smooth. Moreover, outside $\Sigma$ the
energy can be differentiated, so we can use pseudo gradient vector
field construction and we find the wanted retraction (for all
details see \cite[Lemma 5]{Ghi05})
\end{proof}

From Lemma \ref{retrsigma} and Lemma \ref{retru} we get the
following result, which states that Definition \ref{ind} is a good
definition.

\begin{teo}
\label{main} Let $M$ be a conical manifold, let $p,q \in M$ be
s.t. $E$ admits only a finite set of critical points under a
certain level $\bar c$. Let $c<\bar c$ be a critical level. There
exists an ${\varepsilon}>0$ s.t., for any
${\varepsilon}_1,{\varepsilon}_2<{\varepsilon}$, we have that
\begin{equation}
{
P}_\lambda({\Omega}^{c+{\varepsilon}_1},{\Omega}^{c-{\varepsilon}_1})=
{
P}_\lambda({\Omega}^{c+{\varepsilon}_2},{\Omega}^{c-{\varepsilon}_2}).
\end{equation}
\end{teo}

\begin{proof}
Because there are a finite number of critical points, we can
choose ${\varepsilon}$ s.t. $[c-{\varepsilon},c+{\varepsilon}]$
contains only $c$ as critical value. Chosen
${\varepsilon}_1<{\varepsilon}_2<{\varepsilon}$, we want to
retract ${\Omega}^{c+{\varepsilon}_2}$ on
${\Omega}^{c+{\varepsilon}_1}$ and ${\Omega}^{c-{\varepsilon}_1}$
on ${\Omega}^{c-{\varepsilon}_2}$. Defined as before $\Sigma$ and
$\Sigma_i$ in ${\Omega}^{c+{\varepsilon}_2}_{c+{\varepsilon}_1}$,
we note that for $i\neq j$ then $\Sigma_i\cap\Sigma_j=\emptyset$,
because the geodesics $\gamma_i$ and $\gamma_j$ are distinct. Set
$N$ the number of critical points under the level
$c+{\varepsilon}$, we can find a neighborhood $R\supset
\bigcup\limits_i\Sigma_i$ and a retraction $\eta_{R}$ as in Lemma
\ref{retrsigma}. Furthermore, for every
$U\supset\bigcup\limits_{i=0}^N\Sigma_i$ there exists a retraction
$\eta_U$ on
${\Omega}^{c+{\varepsilon}_2}_{c+{\varepsilon}_1}\smallsetminus U$
in analogy with Lemma \ref{retru}. We choose $U$ s.t. there exists
$V$ a neighborhood of $\bigcup\limits_i\Sigma_i$ with
\begin{displaymath}
\bigcup_i\Sigma_i\subsetneq U\subsetneq V\subsetneq R.
\end{displaymath}
For the sake of simplicity we will suppose that $\eta_U$ and
$\eta_R$
 are defined for $0\leq t \leq 1$ and that $\nu$
is the same for both of them. Let $\theta_1:
{\Omega}^{c+{\varepsilon}_2}\rightarrow[0,1]$ be a continuous map
s.t.
\begin{displaymath}
\theta_{1}|_U\equiv0;
\end{displaymath}
\begin{displaymath}
\theta_{1}|_{{\Omega}^{c+{\varepsilon}_1}\smallsetminus V}\equiv1.
\end{displaymath}
Then we define a continuous map
\begin{equation}
\mu_1:{\Omega}^{c+{\varepsilon}_2}\times[0,1]\rightarrow{\Omega}^{c+{\varepsilon}_2},
\end{equation}
\begin{equation}
\mu_1(\beta,t)=\eta_U(\beta,\theta_1(\beta)t);
\end{equation}
we know that $E(\mu_1(\beta,t))-E(\beta)\leq-\nu
t\theta_1(\beta)$, so
\begin{displaymath}
\mu_1({\Omega}^{c+{\varepsilon}_2},1)\subset
V\cup{\Omega}^{c+{\varepsilon}_2-\nu},
\end{displaymath}
in fact if $\mu_1(\beta,t)\notin V$ for all $t$, then
$E(\mu_1(\beta,t))-E(\beta)\leq-\nu t$, so
$\mu_1(\beta,1)\in{\Omega}^{c+{\varepsilon}_2-\nu}$.

By $\mu_1$ we have retracted ${\Omega}^{c+{\varepsilon}_2}$ on
${\Omega}^{c+{\varepsilon}_2-\nu}\cup V$; now we define a
continuous map
$\theta_2:{\Omega}^{c+{\varepsilon}_2}\rightarrow[0,1]$ s.t.
\begin{displaymath}
\theta_{1}|_{{\Omega}^{c+{\varepsilon}_2}_{c+{\varepsilon}_2-\nu/2}}\equiv1;
\end{displaymath}
\begin{displaymath}
\theta_{1}|_{{\Omega}^{c+{\varepsilon}_1-\nu}}\equiv0.
\end{displaymath}

Then set
\begin{equation}
\mu_2:V\cup{\Omega}^{c+{\varepsilon}_2-\nu}\times[0,1]\rightarrow{\Omega}^{c+{\varepsilon}_2},
\end{equation}
\begin{equation}
\mu_2(\beta,t)=\eta_R(\beta,\theta_2(\beta)t);
\end{equation}
$\mu_2$ is a continuous map that retracts
$V\cup{\Omega}^{c+{\varepsilon}_2-\nu}$ on
${\Omega}^{c+{\varepsilon}_2-\nu/2}$. By iterating this algorithm
we can retract continuously ${\Omega}^{c+{\varepsilon}_2}$ on
${\Omega}^{c+{\varepsilon}_1}$.

The retraction of ${\Omega}^{c-{\varepsilon}_1}$ on
${\Omega}^{c-{\varepsilon}_2}$ is obtained in the same way.
\end{proof}
\begin{rem}
The hypothesis of finiteness of critical points seems quite
restrictive. Indeed, because we have a finite number of vertexes,
in many concrete examples this assumption is easily verified, as
shown in the applications.
\end{rem}

To conclude this section we show some properties of ${
P}_\lambda$ useful in the next.
\begin{lemma}
\label{pl} (see \cite[lemma 4.2,(v)]{Ben91}). Let $X$ a metric
space, $A,B$ two closed subspace of $X$ s.t. $B\subset A\subset
X$; then there exist a formal series ${  Q}_\lambda$ s.t.
\begin{equation}
\label{plambda} {  P}_\lambda(X,A)+{  P}_\lambda(A,B)= {
P}_\lambda(X,B)+(1+\lambda){  Q}_\lambda.
\end{equation}
\end{lemma}
\begin{proof}
We consider the exact sequence relative to the triple $X,A,B$:
\begin{equation}
\label{esatta} \ldots \stackrel{\delta^*_{q-1}}{\longrightarrow}
H^q(X,A) \stackrel{i^*_q}{\longrightarrow}H^q(X,B)
\stackrel{j^*_q}{\longrightarrow}H^q(A,B)
\stackrel{\delta^*_q}{\longrightarrow}\ldots
\end{equation}
and we set
\begin{displaymath}
\begin{array}c
a_q=\dim(\ker i^*_q),\\
b_q=\dim(\ker j^*_q),\\
c_q=\dim(\ker \delta^*_q).
\end{array}
\end{displaymath}
By exactness of (\ref{esatta}) we have that
\begin{displaymath}
\begin{array}{rcl}
\dim H^q(X,A)&=&a_q+c_{q-1};\\
\dim H^q(X,B)&=&a_q+b_q;\\
\dim H^q(A,B)&=&b_q+c_q,
\end{array}
\end{displaymath}
thus
\begin{displaymath}
\begin{array}{rcl}
{  P}_\lambda(X,A)&=&\sum_{q=0}^\infty(a_q+c_{q-1})\lambda^q;\\
{  P}_\lambda(X,B)&=&\sum_{q=0}^\infty(a_q+b_q)\lambda^q;\\
{  P}_\lambda(A,B)&=&\sum_{q=0}^\infty(b_q+c_q)\lambda^q.
\end{array}
\end{displaymath}
So
\begin{equation}
{  P}_\lambda(X,A)+{  P}_\lambda(A,B)= {
P}_\lambda(X,B)+\sum_{q=0}^\infty(c_q+c_{q-1})\lambda^q,
\end{equation}
but we can write
\begin{equation}
\sum_{q=0}^\infty(c_q+c_{q-1})\lambda^q=
(1+\lambda)\sum_{q=0}^\infty c_q\lambda^q,
\end{equation}
and then, by setting ${  Q}_\lambda=\sum_{q=0}^\infty
c_q\lambda^q$, we conclude the proof.
\end{proof}
\begin{lemma}
\label{pm} (see \cite[theorem II.3.5]{Ben95}) Let $M$ a conical
manifold. If $M$ is topologically trivial, then also ${\Omega}$ is
topologically trivial, thus
\begin{displaymath}
{  P}_\lambda({\Omega})=1.
\end{displaymath}
\end{lemma}
\begin{proof}
Obvious.
\end{proof}

By lemma \ref{pl} we can formulate the Morse relations for
critical levels.

\begin{teo}
\label{mainval} Let $M$ a conical manifold; let $p,q\in M$ s.t.
$E$ admits only a finite number of critical points for every
sublevel ${\Omega}^c$

Then, if $a$ and $b$ are regular levels there exists a formal
series ${  Q}_\lambda$ such that {\em
\begin{equation}
\label{morseval} \sum_{c\text{ critical in }(a,b)}i(k_c)= {
P}_\lambda({\Omega}^a,{\Omega}^b)+(1+\lambda) {  Q}_\lambda.
\end{equation}
} Moreover, if we consider the whole space ${\Omega}$ there exists
a ${  Q}_\lambda$ such that {\em
\begin{equation}
\label{morseval2} \sum_{c\text{ critical}}i(k_c)= {
P}_\lambda({\Omega})+(1+\lambda) {  Q}_\lambda.
\end{equation}
}
\end{teo}
\begin{proof}
We know that there is a finite set, say $\{c_1,\cdots,c_l\}$, of
critical levels in $[a,b]$. We can, by previous theorem, iterate
(\ref{plambda}) obtaining that
\begin{equation}
\sum_{j=1}^l i(k_{c_j})= {
P}_\lambda({\Omega}^{c_{l+{\varepsilon}}},{\Omega}^{c_{1-{\varepsilon}}})+(1+\lambda)
{  Q}_\lambda.
\end{equation}
Then by deformation lemma we obtain (\ref{morseval}).

Finally, by a limiting process, and considering that
${\Omega}^{-1}=\emptyset$ we can prove (\ref{morseval2}) (we
recall that, by definition, ${  P}_\lambda({\Omega},\emptyset)={
P}_\lambda({\Omega})$).
\end{proof}

\section{Morse Theory For Geodesics}\label{sezpunti}

In this section we will define the Morse index for an isolated
geodesics and finally prove, as claimed, the Morse relations.
Under some {\em a priori\/} bound on the number of geodesics, we
obtain an analogous of the Morse relations that holds in the
smooth case.

\begin{defin}
\label{indgeo} Let $c$ an isolated critical value. Let $k_c$ the
set of $c$ energy geodesics and let $\gamma$ an isolated point of
$k_c$. Then
\begin{equation}
i(\gamma)={  P}_\lambda({\Omega}^c,{\Omega}^c\smallsetminus
\gamma)
\end{equation}
\end{defin}

To proceed, we need to recall the continuity property of
Alexander-Spanier cohomology (see \cite[Cor. 6.6.3]{Spa66}).
\begin{teo}
\label{alex} Let $X\supset A\supset B$, $X$ a paracompact
Hausdorff space, $A,B$ closed in $X$.Then
\begin{equation}
\liminj_{(U,V)\supset(A,B)} H^q(U,V)=H^q(A,B).
\end{equation}
\end{teo}

We can prove now the main theorem of this paper.
\begin{teo}
\label{mainpunt} Let M a conical manifold. let $p,q\in M$ s.t. $E$
admits only a finite set of geodesics for every sublevel
${\Omega}^c$.

Then, if $a,b$ are regular values, we have {\em
\begin{equation}
\label{morsepunt} \sum_{\gamma \in{\Omega}_a^b\text{
geodesic}}i(\gamma)= {
P}_\lambda({\Omega}^a,{\Omega}^b)+(1+\lambda) {  Q}_\lambda.
\end{equation}
} Furthermore, if we consider the whole space we obtain {\em
\begin{equation}
\label{morsepunt2} \sum_{\gamma\text{ geodesic}}i(\gamma)= {
P}_\lambda({\Omega})+(1+\lambda) {  Q}_\lambda.
\end{equation}
}
\end{teo}
\begin{proof}
We know, by theorem \ref{mainval} that
\begin{displaymath}
\sum_{c\text{ critical in }(a,b)}i(k_c)= {
P}_\lambda({\Omega}^a,{\Omega}^b)+(1+\lambda) {  Q}_\lambda;
\end{displaymath}
where
\begin{displaymath}
i(k_c)={
P}_\lambda({\Omega}^{c+{\varepsilon}},{\Omega}^{c-{\varepsilon}}).
\end{displaymath}

Because the critical point are in a finite number, we can choose
$d$ s.t., if $i\neq j$, $B(\gamma_i,d)\cap
B(\gamma_j,d)=\emptyset$. Setting $B_{k_c}(d)=\bigcup_{\gamma\in
k_c} B_\gamma$, we have obviously, by the second deformation
lemma, that ${\Omega}^{c-{\varepsilon}} \simeq
{\Omega}^{c}\smallsetminus B_{k_c}(d)$, thus
\begin{displaymath}
{
P}_\lambda({\Omega}^{c+{\varepsilon}},{\Omega}^{c-{\varepsilon}})=
{ P}_\lambda({\Omega}^{c+{\varepsilon}},{\Omega}^{c}\smallsetminus
B_{k_c}(d)).
\end{displaymath}
Then, by the continuity property of the Alexander-Spanier
cohomology (Theorem \ref{alex}) we have
\begin{displaymath}
{
P}_\lambda({\Omega}^{c+{\varepsilon}},{\Omega}^{c}\smallsetminus
B_{k_c}(d))= { P}_\lambda({\Omega}^{c},{\Omega}^{c}\smallsetminus
B_{k_c}(d)).
\end{displaymath}
To obtain the proof we want to show that
\begin{displaymath}
H^*({\Omega}^{c},{\Omega}^{c}\smallsetminus  B_{k_c}(d))\simeq
H^*({\Omega}^{c},{\Omega}^{c}\smallsetminus  k_c).
\end{displaymath}
It's easy to see that both ${\Omega}^{c}\smallsetminus
B_{k_c}(d)$, ${\Omega}^{c}\smallsetminus  k_c$ are locally
contractible subsets, so there is an isomorphism between
Alexander-Spanier and singular cohomology $\bar H^*$ (see
\cite[Chap.6,Sec.9]{Spa66}). Furthermore, according to Definition
\ref{defpol}, we are working with coefficient in a field ${\mathbb
F}$, so there is also an isomorphism between $\bar H^*$ and
$\Hom(\bar H_*)$ (see \cite[Th.5.5.3]{Spa66}). So it is enough to
show that
\begin{displaymath}
\Hom(H_*({\Omega}^{c},{\Omega}^{c}\smallsetminus
B_{k_c}(d)))\simeq
\Hom(H_*({\Omega}^{c},{\Omega}^{c}\smallsetminus k_c)).
\end{displaymath}
Given $\omega\in H_*({\Omega}^{c},{\Omega}^{c}\smallsetminus k_c)
$, we know that $\omega\in
H_*({\Omega}^{c},{\Omega}^{c}\smallsetminus  B_{k_c}(d)) $ for a
sufficiently small $d$, because $\omega$ has compact support. So
\begin{displaymath}
H_*({\Omega}^{c},{\Omega}^{c}\smallsetminus k_c)= \liminj_d
H_*({\Omega}^{c},{\Omega}^{c}\smallsetminus  B_{k_c}(d)),
\end{displaymath}
but $H_*({\Omega}^{c},{\Omega}^{c}\smallsetminus  B_{k_c}(d))$ is
definitely constant in $d$, thus we have that
\begin{displaymath}
H_*({\Omega}^{c},{\Omega}^{c}\smallsetminus  B_{k_c}(d))=
H_*({\Omega}^{c},{\Omega}^{c}\smallsetminus k_c).
\end{displaymath}
We proved that
\begin{displaymath}
{
P}_\lambda({\Omega}^{c+{\varepsilon}},{\Omega}^{c-{\varepsilon}})=
{ P}_\lambda({\Omega}^c,{\Omega}^{c}\smallsetminus k_c).
\end{displaymath}
Finally, by excision, we obtain
\begin{displaymath}
{  P}_\lambda({\Omega}^{c},{\Omega}^{c}\smallsetminus k_c)=
\bigoplus_{\gamma\in k_c} {
P}_\lambda({\Omega}^{c},{\Omega}^{c}\smallsetminus \gamma)=
\sum_{\gamma\in k_c}i(\gamma)
\end{displaymath}
that proves (\ref{morsepunt}). As usual by a limiting process we
prove (\ref{morsepunt2}).
\end{proof}

\begin{rem}
Following the proof of the above theorem, it's easy to see that,
if $k_c=\{\gamma\}$, then $i(k_c)=i(\gamma)$, so the definition of
index for a geodesic and for a critical set are compatible.
\end{rem}
By the index we can also define the multiplicity of a critical
point.
\begin{defin}\label{mult}
Let $\gamma$ an isolated critical point of energy. Then
\begin{displaymath}
\mult(\gamma)=i(\gamma)|_{\lambda=1}.
\end{displaymath}
\end{defin}

It is well known that the geodesics with high multiplicity are
unstable in smooth manifolds, i.e. up to small perturbations of
extremal points $p$ and $q$ we can find only isolated geodesics,
that have multiplicity 1. For conical manifolds this is false.

We will show in the next example, that in some cases, given $p,q$
on a conical manifold $M$, there exists a geodesic $\gamma_0$ with
$\mult(\gamma_0)>1$ joining them; for a sufficiently small change
of the boundary data $p,q$ we will obtain another geodesic with
the same multiplicity of $\gamma_0$.

\begin{ex}\label{cono}

Let $C$ be a 2-dimensional  Euclidean half cone embedded in $R^3$;
let $p,q$ be two points of $C$ and let $v$ be its vertex. We want
to estimate the number of standard geodesics between $p$ and $q$.

\begin{figure}[h]
\begin{center}
\includegraphics[scale=.8]{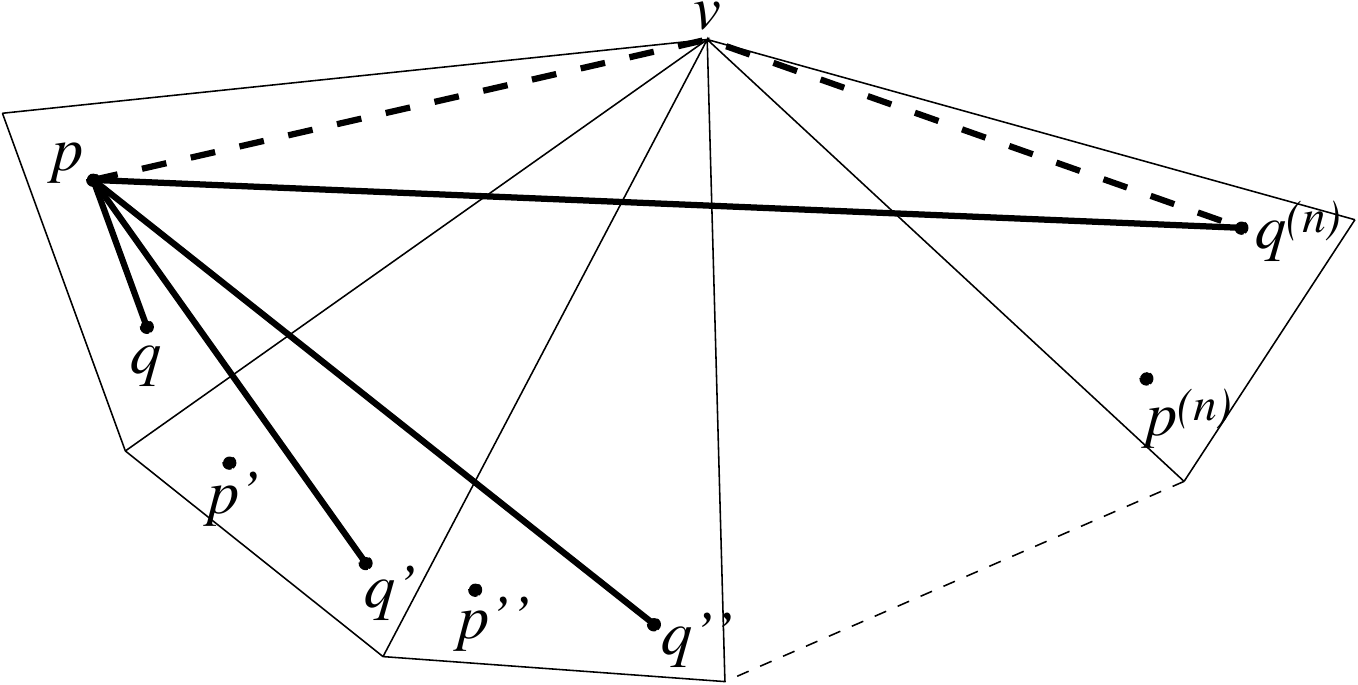}
\caption{Geodesic on a Cone.} \label{Figure2}
\end{center}
\end{figure}

We represent $C$ as a circular sector in the plane: let $\alpha$
be the wideness of the sector; we see that, if $\alpha$ is
sufficiently small, there are many regular geodesic joining $p$
and $q$ (see, for instance \cite[exercise 6, section 4-7]{dC76}),
and the number of these geodesics depends on $p,q$ and $\alpha$.
If $\alpha$ is bigger than $\pi$, on the contrary, we have an
unique geodesic joining $p$ and $q$. We can also observe that all
these curves have different energies, and that they are shorter
than the broken straight line $\gamma_0$, which joins $p$ to $v$
and $v$ to $q$, as shown in Fig \ref{Figure2}. Because the cone is
flat, any one of these geodesics has Morse index 0. In this case
classical Morse theory fails, in fact ${  P}_\lambda=1$, because
the cone is contractible (see lemma \ref{pm}), so
\begin{equation}
\sum_{i=1}^n \lambda^{m(\gamma_i)}=n=1+(1+\lambda){  Q}_\lambda
\end{equation}
and so there will be at least $n-1$ geodesics whose Morse index is
1: this is a contradiction. If use our definition of generalized
geodesic (Definition \ref{defgeod}) to estimate the number of
critical points of energy, we see that we must consider all the
classical geodesics $\gamma_1$,\ldots,$\gamma_n$, plus the broken
geodesic $\gamma_0$ (up to reparametrization).  The critical
levels are all distinct, so every geodesic is isolated and  we can
apply (\ref{morsepunt2}), which states that Morse relations hold.
These estimates give us that
\begin{equation}
n+i(\gamma_0)=1+(1+\lambda){  Q}_\lambda,
\end{equation}
then
\begin{equation}
{  Q}_\lambda=n-1+{  R}_\lambda,
\end{equation}
where ${  R}_\lambda$ is a formal series with non negative
coefficients.

So
\begin{equation}
i(\gamma_0)\geq(n-1)\lambda,
\end{equation}
and
\begin{equation}
\mult(\gamma_0)\geq(n-1).
\end{equation}
As claimed, we have a geodesic $\gamma_0$ with an high
multiplicity. If we look at the definition of this curve, we have
obviously that $\mult(\gamma_0)$ is stable for small perturbations
of $p$ and $q$.
\end{ex}

The above example gives us also a geometrical interpretation of
multiplicity, in fact, if we smooth $C$, classical Morse theory
holds, and we will see exactly $\mult(\gamma_0)$ geodesics
appearing. We can think that, is some sense, this geodesics
accumulates in $\gamma_0$.

Furthermore  $\mult(\gamma_0)$ depends on the wideness $\alpha$
and it is zero if $\alpha>\pi$. We have a change of topology of
the energy sublevels only when $\mult(\gamma_0)\neq 0$; this fact
allows us another consideration that we clarify with the following
example.

\begin{ex}\label{piano}
Let $M=\R^2$ be endowed with the usual scalar product and let
$v=(0,0)$, $p=(0,1)$, $q=(1,0)$; obviously, the unique Riemannian
geodesic between $p$ and $q$ is the straight line $pq$.

We can also consider $M$ as a conical manifold with vertex $v$,
according to definition \ref{defvarieta}. We see that also the
broken line $pvq$ is a geodesic following definition
\ref{defgeod}: this fact seems artificial. By Morse relations we
prove that $i(pvq)=0$, in fact, $i(pq)=1$, and ${  P}_\lambda=1$,
thus
\begin{equation}
i(pvq)+1=1+(1+\lambda){ Q}_\lambda.
\end{equation}
Then even $\mult(pvq)=0$. The geodesic $pvq$, thus, does not
affect the structure of the sublevels of energy: the multiplicity
of a curve allows us to distinguish the geodesics as $pvq$ from
the ones that are meaningful critical points of energy.
\end{ex}
This example shows that, by the multiplicity, we also recover the
geometrical meaning of geodesic that we lost in the starting
definitions.

\section{Comparison with weak-slope approach}\label{sezconfronto}

To complete the our study about geodesics, we want to compare our
approach with Degiovanni-Marzocchi-Morbini one. These authors in
\cite{DM99,MM02} study the problem of finding geodesics in
manifolds with boundary by using the weak slope theory; under
certain hypothesis on the conical manifold, both approaches are
possible. We'll see that there is a strong bound between these
theories, and we'll be able to recover a regularity result. Let
$M$ be a conical manifold embedded in $\R^n$, and assume that
$M$ is the Lipschitz boundary of an open convex set $A\subset
\R^n$. Let $C=\R^n\smallsetminus A$, and let $p,q\in
M\subset C$. We set the functional space
\begin{displaymath}
X=\{\gamma\in H^1([0,1],C),\gamma(0)=p,\gamma(1)=q\},
\end{displaymath}
and we define a functional
\begin{equation}
\begin{array}c
J:L^2([0,1],\R^n)\rightarrow \R\cup \{\infty\},\\
\\
J(\gamma)=\left\{
\begin{array}{ll}
\int_0^1|\gamma'|^2&\text{ if }\gamma\in X\\
+\infty&\text{ otherwise},
\end{array}
\right.
\end{array}
\end{equation}
in order to have the P.S. condition for every sublevel $J^c$.
\begin{defin}
A curve $\gamma$ is a critical point of $J$ iff $|dJ|(\gamma)=0$,
where $|dJ|$ is the weak slope introduced in \cite{DM94}.
\end{defin}
We have that
\begin{teo}\label{regol}
Let $\gamma\in X$ be a critical point for $J$. Then
\begin{itemize}
\item{$|\gamma'|$ is constant almost everywhere.} \item{$\forall
(a,b)\subset[0,1]$ we have that $\gamma'\in BV((a,b),\R^n)$ and
there exists a finite Borel measure $\mu$ on $(a,b)$ and a bounded
Borel function $\nu:(a,b)\rightarrow\R^n$ s.t $\nu(s)\in
N_{\gamma(s)}C$ (the Clarke normal cone to $C$) for $\mu$ a.e.
$s\in(a,b)$ such that
\begin{displaymath}
\int_a^b|\gamma'|ds=-\int_a^b\nu\delta d\mu
\end{displaymath}
$\forall \delta\in W^{1,1}((a,b),\R^n)$ (i.e. in a
distributional sense
 $\gamma''=\nu d\mu$) }
\end{itemize}
\end{teo}
\begin{proof}
This theorem is a reformulation of \cite[th.3.5]{DM99}
and\cite[th.2.10]{MM02}, which we refer to.
\end{proof}

We can prove now the following proposition, that states a link
between our approach and weak slope one.
\begin{prop}
Let $M$ be a conical manifold as above, and let $p,q\in M$. Let
$E$ be defined as usual in (\ref{defen2}). Suppose that exist a
$\gamma\in {\Omega}$ and an ${\varepsilon}>0$ s.t. $\gamma$ is the
unique geodesic (in the sense of Definition \ref{defgeod}) in the
strip ${\Omega}^{c+{\varepsilon}}_{c-{\varepsilon}}$, where
$c=E(\gamma)$.

If $i(\gamma)\neq 0$, then $\gamma$ is a critical point of $J$ in
sense of weak slope, i.e. $|dJ|(\gamma)=0$. Furthermore $\gamma$
has the regularity properties of theorem \ref{regol}.
\end{prop}
\begin{proof}
Obviously  $J(\gamma)=E(\gamma)=c$. Given ${\varepsilon}$ as in
the hypothesis, we know that ${\Omega}^{c+{\varepsilon}}\not\simeq
{\Omega}^{c-{\varepsilon}}$. We can easily see that, by a
convexity argument, $J^a\simeq {\Omega}^a$ for all $a\in\R$:
thus, it must exist a $\beta\in
J^{c+{\varepsilon}}_{c-{\varepsilon}}$ s.t. $|dJ|(\beta)=0$ and,
again by convexity, $\beta([0,1])\subset M$. We apply theorem
\ref{regol} and we obtain that $|\beta'|$ is constant a.e., then
$\beta$ is a generalized geodesics according to Definition
\ref{defgeod}. Also, $\beta\in
{\Omega}^{c+{\varepsilon}}_{c-{\varepsilon}}$: by hypothesis,
$\gamma$ is the unique geodesic in
${\Omega}^{c+{\varepsilon}}_{c-{\varepsilon}}$, so $\beta=\gamma$.
The proof follows immediately.
\end{proof}

\section{A final remark}

We have introduced this theory to solve a problem arising in a
natural way: {\em what is a geodesic on a cone?\/} Although the
answer is obvious for minimal ones, a general definition is
difficult. To overrides this obstacle, we have introduced a large
class of generalized geodesics in Definition \ref{defgeod}.

At a first sight our definition seems a little bit artificial, but
the existence of a Morse theory and the examples provided in the
last section, suggest that this definition is appropriate to our
purpose. Indeed, when a geodesic has a positive index then there
is a change in the topology of energy sublevels: these generalized
geodesics can be considered as critical points, at least in a
topological sense. This fact has two peculiarities.

First, a generalized geodesic is suitable for the global
variational approach.

Second, we have a criterion to establish when a generalized
geodesic is geometrically meaningful. In fact, even if our
definition might introduce "artificial" geodesics, as we have
shown in Example \ref{piano}, the index theory states that these
curves are not critical points of energy; namely, they have 0
multiplicity.

Furthermore, if $M$ is a conical manifolds which is a Lipschitz
boundary of a convex open set in $\R^n$, a geodesic $\gamma$
with non zero index is a critical point even according to
Degiovanni, Marzocchi and Morbini approach.
\nocite{dC92}
\nocite{DM94}

\end{document}